\documentclass[11pt,reqno]{amsart}
\usepackage{amsthm, amsfonts, amssymb, color}
 \usepackage{mathrsfs}
\usepackage{enumerate}
\usepackage{amsmath}
 \usepackage{amstext, amsxtra}
  \usepackage{txfonts}
 \usepackage[colorlinks, linkcolor=black, citecolor=blue, pagebackref, hypertexnames=false]{hyperref}
\usepackage{graphicx}

 \allowdisplaybreaks
\newtheorem{theorem}{Theorem}[section]

\newtheorem{corollary}[theorem]{Corollary}
\newtheorem{lemma}[theorem]{Lemma}
\newtheorem{definition}[theorem]{Definition}

\newtheorem{remark}[theorem]{Remark}

\renewcommand\Re{\operatorname{Re}}
\renewcommand\Im{\operatorname{Im}}

\numberwithin{equation}{section}
\theoremstyle{definition}

\title
 [Observability inequalities]
{Uncertainty principle, minimal escape velocities and observability inequalities for schr\"{o}dinger equations}
\author{Shanlin Huang, Avy Soffer}

\address{ Shanlin Huang,  School of Mathematics and Statistics, Huazhong University of Science and Technology,  Wuhan,  430074, P.R.China  }
\email{shanlin\_huang@hust.edu.cn}

\address{ Avy Soffer, School of Mathematics and Statistics, Central China Normal University, Wuhan, 430079, P. R. China }

\address{Department of Mathematics, Rutgers University, Piscataway, 08854-8019, USA }
\email{soffer@math.rutgers.edu}

\subjclass[2000]{}
\keywords{}

\begin{document}

\maketitle
\begin{abstract}
 We develop a new abstract derivation of the observability inequalities at two points in time for Schr\"{o}dinger type equations. Our approach consists of two steps. In the first step we prove a Nazarov type uncertainty principle associated with a non-negative self-adjoint operator $H$ on $L^2(\mathbb{R}^n)$. In the second step we use results on asymptotic behavior of $e^{-itH}$, in particular, minimal velocity estimates introduced by Sigal and Soffer. Such observability inequalities are closely related to unique continuation problems as well as controllability for the Schr\"{o}dinger equation.

\end{abstract}


\section{Introduction}
In a recent paper by Wang, Wang and Zhang \cite{WWZ}, they established a new type of observability inequality at two points in time for the free Schr\"{o}dinger equation. More precisely, let $u(x, t)$ satisfy
\begin{align}\label{equ1.1}
\begin{cases}
i\partial_{t}u(x, t)+\Delta u(x, t)=0,\,\,\,(x, t)\in \mathbb{R}^n\times (0, \infty),\\
u(x, 0)=u_{0}\in L^2(\mathbb{R}^n).
\end{cases}
\end{align}
Then given any $r_1,\,r_2>0$, and $t_2>t_1\ge 0$, there exists a positive constant $C$ depending only on $n$ such that
\begin{align}\label{equ1.2}
\int_{\mathbb{R}^n}|u_0|^2\,dx\leq Ce^{C\frac{r_1r_2}{t_2-t_1}}\left(\int_{|x|\ge r_1}|u(x, t_1)|^2\,dx+\int_{|x|\ge r_2}|u(x, t_2)|^2\,dx\right), \,\,u_0\in L^2(\mathbb{R}^n).
\end{align}
The proof in \cite{WWZ} is based on the fact that in the free case,  one has the identity
\begin{align}\label{equ1.3}
(2it)^{\frac{n}{2}}e^{-i|x|^2/4t}u(x, t)=\widehat{e^{-i|\cdot|^2/4t}u_0}(x/2t),\,\,\, \text{for all}\,\,t>0,
\end{align}
where $\widehat{\cdot}$ denotes the Fourier transform. After applying \eqref{equ1.3} with a scaling argument, it's easy to see that the estimate \eqref{equ1.2} is closely related to the following Nazarov's uncertainty principle built up in \cite{Jam}: If $A, B$ are subsets of $\mathbb{R}^n$ of finite measure, then
\begin{align}\label{equ1.4}
\int_{\mathbb{R}^n}|f(x)|^2\,dx\leq C(n, A, B)(\int_{\mathbb{R}^n\setminus A}|f(x)|^2\,dx+\int_{\mathbb{R}^n\setminus B}|\hat{f}(\xi)|^2\,d\xi), \,\,\,\,f\in L^2(\mathbb{R}^n),
\end{align}
with
\begin{align}\label{equ1.3.1}
C(n, A, B)\triangleq Ce^{C\min\{|A||B|,~ |B|^{1/n}\omega(A),~ |A|^{1/n}\omega(B)\}},
\end{align}
where $\omega(A)$ denotes the mean width of $A$ (in fact, \eqref{equ1.4} implies \eqref{equ1.2} by choosing $A=\{x\in\mathbb{R}^n:\, |x|\leq r_1\}$ and $B=\{x\in\mathbb{R}^n:\, |x|\leq r_2/(t_2-t_1)$).
\par
A natural question is whether such kind of observability inequalities still hold for more general Hamiltonian. We mention that the approach in \cite{WWZ} is restricted to the free Laplacian, since the argument there is essentially relying on the formula \eqref{equ1.3}, which in turn follows from the fundamental solution  of $e^{it\Delta}$. For general $H$, no such explicit solutions are available, thus one needs to proceed differently.

The motivation of this paper is to develop an abstract approach to obtain observability inequalities at two points in time for $e^{-itH}$ under some general assumptions on $H$. Then we apply it to special cases including Schr\"{o}dinger equation with potentials and fractional Schr\"{o}dinger equations.

We first point out that \eqref{equ1.2} may fail if $H$ has eigenvalues. Indeed if $H\phi=\lambda\phi$, for some $\lambda\in \mathbb{R}$ and $\phi\in L^2$. Then $u(x, t)=e^{-i\lambda t}\phi(x)$ is a solution of the following Cauchy problem
$$
i\partial_{t}u =H u, \qquad u(0,x)=\phi(x)\in L^2(\mathbb{R}^n).
$$
After choosing $r_1=r_2=\sqrt{t_2-t_1}$ in \eqref{equ1.2}, we find that the RHS of \eqref{equ1.2} is equal to $C\int_{|x|\ge \sqrt{t_2-t_1}}{|\phi|^2\,dx}$ with some fixed constant $C$, which goes to zero as $t_2-t_1\rightarrow \infty$. Hence estimate \eqref{equ1.2} can't hold for such $\phi$. Therefore, we only expect \eqref{equ1.2} to hold for vectors lying in the continuous subspace of $H$.

We proceed to illustrate the key idea and main tools used in our approach. To simplify matters, we change the uncertainty principle \eqref{equ1.4} into a form concerning two projection operators on $L^2$, i.e., for any $r>0$, there exists some uniform constant $C>0$ such that
\begin{align}\label{equ1.5}
\|f\|_2^2\leq C\left(\|\chi(|x|\ge r)f\|_2^2+\|\chi(-\Delta\ge r^{-2})f\|_2^2\right), \,\,\,\,f\in L^2(\mathbb{R}^n).
\end{align}
where  $\chi(\Omega)$ stands for the characteristic function of $\Omega$. We mention that inequality \eqref{equ1.5}  indicates that if the initial data is localized in a ball, then its "energy" must have a positive lower bound. Actually, it's easy to see that the estimate \eqref{equ1.5} is equivalent to the following
$$
\|\chi(|x|\le r)f\|\leq C\|\chi(H\ge r^{-2})f\|,\,\,\,\text{for any}\,\, r>0.
$$

Having established this type of uncertainty principle for $H$, we can use propagation estimates, in particular {\bf minimal velocity estimates} to further study the asymptotic behavior of $e^{-itH}f$. To provide intuition in understanding of this method, let us consider the simple case $H=-\Delta$, and assume $f$ is a Schwartz function such that $f\in Ran\, \chi(H\ge \delta )$ with some $\delta>0$, thus $\hat{f}$ is smooth and $\text{supp}\,\hat{f}\subset \{\xi\in \mathbb{R}^n,\,|\xi|\ge \sqrt{\delta}\}$. Then a integration by parts argument yields
$$
\int_{\frac{|x|}{t}<\sqrt{\delta}}{|e^{-itH}f|^2\,dx}=O(t^{-m}),\,\,\,\text{as}\,\, t\rightarrow\infty
$$
for any $m>0$. In this sense, the evolution $e^{-itH}f$ is said to have a minimal  velocity $v_{min}=\sqrt{\delta}>0$. Roughly speaking, the goal of minimal velocity estimates is to obtain similar results for general Hamiltonian via an abstract way. And it's  based on choosing observable (self-adjoint operator) $A$ so that the commutator $i[H, A]$ is positive definite, see section \ref{sec2.2} for a further discussion. Such estimates are crucial in our proof, since it provides quantitative information about the rate with which the wave  $e^{-itH}f$ moves out to spacial infinity. As a comparison, we recall that the RAGE theorem (see e.g. \cite{RS}) indicates that for certain Schr\"{o}dinger operators $H=-\Delta+V$, $e^{-itH}f$ is escaping any fixed ball in a mean ergodic sense:
$$
\lim_{T\rightarrow\infty}\frac{1}{T}\int_0^T{\,dt}\int_{|x|\leq R}{|e^{-itH}f|^2\, dx}=0.
$$
We mention that minimal velocity estimates were first appeared in the work of Sigal and Soffer \cite{SS}, which turned out to be very useful in scattering theory and theory of resonances, we refer to \cite{Ski,SS90, SW, HSS} and references therein for further extensions and applications. One of the novelties in our paper is that we establish the close relationship between observability inequalities and  minimal velocity estimates for Schr\"{o}dinger type equations.

Now we turn to some applications. As is pointed out in \cite{WWZ}, the estimate \eqref{equ1.2} can be used to derive controllability for Schr\"{o}dinger equations. It is also closely related to quantitative unique continuation problems for Schr\"{o}dinger equations. In Sect. \ref{sec4}, we shall use observability inequalities built up in this paper to obtain results concerning unique continuation properties of Schr\"{o}dinger equations with potentials as well as fractional Schr\"{o}dinger equations. Such kind of results for certain linear and nonlinear Schr\"{o}dinger equations were considered by many people, see e.g. \cite{B,KS,IK,IK04,KPV,Zhang}  and references therein.

To our best knowledge, the uniqueness result concerning the fractional Schr\"{o}dinger equation obtained in this paper (see Corollary \ref{thm4.2}) is new.  We also mention that the method in this paper presents an alternative approach to prove unique continuation theorems, see Remark \ref{rmk4.3} for a detailed discussion on previous related results.  For the uncertainty principle and unique continuation inequalities for Schr\"{o}dinger equations, we would like to refer a series of paper by Escauriaza, et al. \cite{EKPV1,EKPV2,EKPV3,EKPV4}.

The rest of the paper is organized as follows. Section \ref{sec2} is divided into two subsections, where we discuss the related uncertainty principle and minimal velocity estimates.  In Sect. \ref{sec3.1}, the observability inequalities are proved based on tools established in Sect.  \ref{sec2}. Furthermore we show in Sect. \ref{sec3.2} that the observability inequalities may not hold by observing the solution at two different points in time, one time in a ball, while the other outside a ball. Section \ref{sec4} is devoted to applications to unique continuation property as well as controllability for the Schr\"{o}dinger equation.
%
%

\section{Main tools}\label{sec2}
\subsection{Uncertainty principle}\label{sec2.1}
In this subsection, we consider generalizations of both \eqref{equ1.4} and \eqref{equ1.5}. We first present an abstract version of the Nazarov type uncertainty principle for a general non-negative self-adjoint operator $H$ on $L^2(\mathbb{R}^n)$, assuming only $L^2-L^{\infty}$ decay estimates of the corresponding heat semigroup $e^{-tH}$. Then we turn to concrete examples including Schr\"{o}dinger operators $H=-\Delta+V$, as well as the fractional Laplacian ($(-\Delta)^{\frac{s}{2}}$, $s>0$). In these cases, we are able to obtain more quantitative results, which will be used in Sect. \ref{sec3}.
\subsubsection{Uniform estimates}\label{sec2.1.1}

\begin{lemma}\label{lem2.1}
Let $H$ be a non-negative self-adjoint operator on $L^2(\mathbb{R}^n)$, $n\ge 1$. Assume that there is some $\gamma>0$ such that
\begin{align}\label{equ2.1}
\|e^{-tH}\|_{L^2\rightarrow L^{\infty}}\leq Ct^{-\gamma}.
\end{align}
Then for any $R>0$, there is a constant $\delta_R>0$, such that for all $f\in L^2(\mathbb{R}^n)$
\begin{align}\label{equ2.2}
\|f\|^2\leq C\left(\|\chi(|x|\ge R)f\|^2+\|\chi(H\ge \delta_R)f\|^2\right).
\end{align}
\end{lemma}
\begin{proof}
We first point out that \eqref{equ2.2} is equivalent to proving that there is some $C>0$ such that for any $R>0$,
\begin{align}\label{equ2.3}
\|\chi(|x|\leq R)f\|^2\leq C\|\chi(H\ge \delta_R)\chi(|x|\leq R)f\|^2, \,\,\,\,f\in L^2,
\end{align}
which in turn is  equivalent to prove that for any $R>0$, there exists a constant $\delta_R>0$,
\begin{align}\label{equ2.4}
\|\chi(|x|\leq R)\chi(H\le \delta_R)\|_{L^2\rightarrow L^2}<1.
\end{align}
In order to prove \eqref{equ2.4}, we note that it follows from \eqref{equ2.1} and Laplace transform
\begin{align*}
\|(H+\epsilon_0)^{-\alpha}\|_{L^2\rightarrow L^{\infty}}\leq& \frac{1}{\Gamma(\alpha)}\int_0^{\infty}\|e^{-tH}\|_{L^2\rightarrow L^{\infty}}e^{-t\epsilon_0}t^{\alpha-1}\,dt\\
\leq& \frac{1}{\Gamma(\alpha)}\int_0^{\infty}e^{-t\epsilon_0}t^{\alpha-1-\gamma}\,dt\\
\leq& C\epsilon_0^{\gamma-\alpha},
\end{align*}
provided $\alpha>\gamma$. Hence if we denote by $K(x, y)$ the kernel of the operator $\chi(|x|\leq R)(H+\epsilon_0)^{-\alpha}$, we deduce that
\begin{align*}
\int{|K(x, y)|^2\,dxdy}\leq& \int_{\mathbb{R}^n}{\chi(|x|\leq R)\,dx}\int_{\mathbb{R}^n}{|(H+\epsilon_0)^{-\alpha}(x, y)|^2\,dy}\\
\leq& CR^n\epsilon_0^{2(\gamma-\alpha)},
\end{align*}
which implies that $\chi(|x|\leq R)(H+\epsilon_0)^{-\alpha}$ is a Hilbert-Schmitd operator and furthermore
\begin{align}\label{equ2.5}
\|\chi(|x|\leq R)(H+\epsilon_0)^{-\alpha}\|_{L^2\rightarrow L^2}\leq CR^{\frac n2}\epsilon_0^{\gamma-\alpha}.
\end{align}
Therefore for any $R>0$, we can choose $\epsilon_0>0$ small enough, and then let $\delta_R=\epsilon_0$, we obtain
\begin{align*}
\|\chi(|x|\leq R)\chi(H\le \delta_R)\|\leq& \|\chi(|x|\leq R)(H+\epsilon_0)^{-\alpha}\|\cdot\|(H+\epsilon_0)^{\alpha}\chi(H\le \delta_R)\|\\
\leq& CR^{\frac n2}\cdot\epsilon_0^{\gamma-\alpha}\cdot(\delta_R+\epsilon_0)^{\alpha}\\
\leq& CR^{\frac n2}\epsilon_0^{\gamma}<1,
\end{align*}
which proves \eqref{equ2.4}.
\end{proof}

We mention that the result above doesn't imply sharp relationship between $R$ and $\delta_R$. However, in the case $H=-\Delta+V$ with suitable class of potentials, instead of using the heat kernel estimate \eqref{equ2.1}, we shall establish sharp results for $n\ge 3$  by relying on the limiting behavior of $(H+\epsilon)^{-1}$ as $\epsilon\rightarrow 0$.
We first recall the definition of Kato class and the related global Kato norm.

\begin{definition}
Let $n\ge 3$, a real measurable function $V(x)$ is said to lied in the Kato class if
$$
\lim_{\delta \to 0}\sup_{x\in\mathbb{R}^n}\int_{|x-y|<\delta}{\frac{|V(y)|}{|x-y|^{n-2}}\,dy}=0.
$$
Moreover, the global Kato norm of $V(x)$ is defined as
$$
\|V\|_K=\sup_{x\in\mathbb{R}^n}\int_{\mathbb{R}^n}{\frac{|V(y)|}{|x-y|^{n-2}}\,dy}.
$$
\end{definition}

Our assumption on $V$ is the following
\begin{align}\label{equ2.6}
\begin{cases}
V=V_+-V_-,\,\,\, \text{where}\,\,V_+=\max\{V, 0\}\\
V_+\,\,\text{is of Kato class},\,\,\, \|V_-\|_K<4\pi^{\frac{n}{2}}/\Gamma(\frac{n}{2}-1).
\end{cases}
\end{align}

It's known that (see \cite[Lemma 3.1]{DP}) under this assumption, $-\Delta+V$ defined on $C_0^{\infty}(\mathbb{R}^n)$ extends to a unique nonnegative self-adjoint operator. Furthermore, we shall prove
\begin{lemma}\label{lem2.2}
Let $n\ge 3$ and assume that $V$ satisfies condition \eqref{equ2.6}. Then for any $R>0$, there are uniform constants $C, \delta>0$ such that
\begin{align}\label{equ2.7}
\|f\|^2\leq C\left(\|\chi(|x|\ge R)f\|^2+\|\chi(H\ge \delta R^{-2})f\|^2\right),\,\,\, f\in  L^2(\mathbb{R}^n).
\end{align}
\end{lemma}
\begin{proof}
We first point out that it suffices to prove the case $R=1$ via a scaling argument. To this end, we consider the scaling operator
$$
U_Rf=R^{\frac n2}f(R\cdot),\,\,\,\,\, R>0.
$$
Clearly, $U_R$ is an isometry on $L^2(\mathbb{R}^n)$. Now set $H_R=\Delta+V_R$, where $V_R=R^2V(R\cdot)$. A direct computation yields
\begin{align}\label{equ2.8}
U_R^{-1}H_RU_R=R^2H.
\end{align}
Thus \eqref{equ2.7} follows by proving that for all $R>0$, there exists a $\delta>0$ such that
\begin{align}\label{equ2.9}
\|\chi(|x|\leq 1)\chi(H_R\leq \delta)\|_{L^2\rightarrow L^2}<1.
\end{align}
In order to show \eqref{equ2.9}, we point out that the key is to verify the following
\begin{align}\label{equ2.10}
\||x|^{-1}H_R^{-\frac12}\|_{L^2\rightarrow L^2}\leq C_H<\infty.
\end{align}
Indeed, applying \eqref{equ2.10}, we have
\begin{align}\label{equ2.10'}
\|\chi(|x|\leq 1)\chi(H_R\leq \delta)\|\leq& \|\chi(|x|\leq 1)|x|\|\cdot\||x|^{-1}H_R^{-\frac12}\|\cdot\|H^{\frac12}\chi(H_R\leq \delta)\|\nonumber \\
\leq& C_H\cdot\delta^{\frac12}<1,
\end{align}
provided $\delta<C_H^{-2}$, which implies \eqref{equ2.9}. Therefore it remains to prove \eqref{equ2.10}. Note that when $n\ge 3$, in view of the Hardy's type inequality
\begin{align}\label{hardy}
\||x|^{-1}(-\Delta)^{-\frac12}\|_{L^2\rightarrow L^2}\leq C.
\end{align}
\eqref{equ2.10} follows if one can prove that
\begin{align}\label{equ2.11}
\|(-\Delta)^{\frac12}H_R^{-\frac12}\|_{L^2\rightarrow L^2}\leq C\,\,\,  \text{for all}\,\,R>0.
\end{align}
In order to prove \eqref{equ2.11}, we denote the operator
$$
T_V=|V_{R-}|^{\frac12}(-\Delta)^{-\frac12},
$$
then by a $TT^*$ argument (see \cite[Lemma 3.1]{DP}) and observing that the global Kato norm is invariant under the scaling, i.e., $\|V_R\|_K=\|V\|_K$, for any $R>0$, we deduce that
\begin{align*}
\|T_VT_V^*f\|^2\leq& \frac{\|V_{R-}\|_K}{\alpha_n^2}\int\int{\frac{|V_{R-}||f(y)|^2}{|x-y|^{n-2}}\,dxdy} \\
\leq& \frac{\|V_{R-}\|^2_K}{\alpha_n^2}\|f\|^2\\
=& \frac{\|V_{-}\|^2_K}{\alpha_n^2}\|f\|^2,
\end{align*}
where in the first inequality, we used the fact that the kernel of $(-\Delta)^{-1}$ satisfies
$$
|(-\Delta)^{-1}(x,y)|\leq \frac{1}{\alpha_n|x|^{n-2}},\,\,\,\,\, \alpha_n=4\pi^{\frac n2}/\Gamma(\frac n2-1).
$$
\end{proof}
The method in the proof of Lemma \ref{lem2.2} can be applied to other situations directly, such as $H=-\Delta-\frac{c_n}{|x|^2}$, $c_n<\frac{(n-2)^2}{4}$ and $H=(-\Delta)^{\alpha}$, $\alpha>0$. More precisely we have
\begin{corollary}
Let $n\ge 3$,  $H=-\Delta-\frac{c_n}{|x|^2}$, where $c_n<\frac{(n-2)^2}{4}$. Then  Then for any $R>0$, there are uniform constants $C, \delta>0$ such that
\begin{align}\label{equ2.12}
\|f\|^2\leq C\left(\|\chi(|x|\ge R)f\|^2+\|\chi(H\ge \delta R^{-2})f\|^2\right),\,\,\, f\in  L^2(\mathbb{R}^n).
\end{align}
\end{corollary}
\begin{proof}
We follow the proof in Lemma \ref{lem2.2}, and note that in this case, we have
$$
H_R=H,\,\,\,\, \text{for any}\,\, R>0.
$$
Meanwhile, it follows from the Hardy's inequality \eqref{hardy} that
$$
(\varphi, \frac{c_n}{|x|^2}\varphi)\leq \|\nabla\varphi\|^2,\,\,\,\,\, c_n<\frac{(n-2)^2}{4},
$$
which implies that
$$
\|(-\Delta)^{\frac12}H^{-\frac12}\|_{L^2\rightarrow L^2}\leq C.
$$
\end{proof}

\begin{corollary}\label{cor2.5}
Let $n\ge 3$,  $H=(-\Delta)^{\frac{s}{2}}$, $s>0$. Then  Then for any $R>0$, there are uniform constants $C, \delta>0$ such that
\begin{align}\label{equ2.14}
\|f\|^2\leq C\left(\|\chi(|x|\ge R)f\|^2+\|\chi(H\ge \delta R^{-s})f\|^2\right),\,\,\, f\in  L^2(\mathbb{R}^n).
\end{align}
\end{corollary}
\begin{proof}
We note that the two types of estimates \eqref{equ2.2} and \eqref{equ2.4} are equivalent to each other. Thus the proof is followed from
\begin{align*}
\|\chi(|x|\leq R)\chi(H\leq \delta R^{-s})\|\leq& \|\chi(|x|\leq R)|x|\|\cdot\||x|^{-1}(-\Delta)^{-\frac12}\|\cdot\|H^{\frac{1}{s}}\chi(H\leq \delta R^{-s})\| \\
\leq& C_H\cdot\delta^{\frac{1}{s}}<1,
\end{align*}
provided $\delta<C_H^{-s}$.
\end{proof}

We mention that the uniform estimates \eqref{equ2.7}, \eqref{equ2.14} are consistent with the Uncertainty principle \eqref{equ1.4} (or \eqref{equ2.1.1} below) in the following sense: they correspond  to $r_1=R$, $r_2=R^{-1}$, then the constant $C(n, R, R^{-1})$ is independent of $R$ in view of \eqref{equ1.3.1} (or \eqref{equ2.1.2} below). In the following part, we shall assume that $r_1,\,r_2>0$ are independent, and discuss the explicit dependence of the constant on $r_1, r_2$ for general $H$.

\subsubsection{Explicit constants}\label{sec2.1.2}
We recall that in the original Nazarov type Uncertainty principle \eqref{equ1.4}, if we let $A=B(0, r_1),\,\,\text{and}\,\,B=B(0, r_2)$ with arbitrary $r_1,\,r_2>0$, then it could be rewritten as
\begin{align}\label{equ2.1.1}
\|f\|_2^2\leq C(n, r_1, r_2)\left(\|\chi(|x|\ge r_1)f\|_2^2+\|\chi(-\Delta\ge r^2_2)f\|_2^2\right), \,\,\,\,f\in L^2(\mathbb{R}^n),
\end{align}
with
\begin{align}\label{equ2.1.2}
C(n, r_1, r_2)\leq Ce^{c_1\min\{(r_1r_2)^n,\, r_1r_2\}},
\end{align}
which is followed from \eqref{equ1.3.1} by noting that when $A$ is a ball, then $\omega(A)$ is the diameter (see \cite[p.36]{Jam}).

We shall establish inequalities of this form with $H_0=-\Delta$ replaced by the Schr\"{o}dinger operator $H=-\Delta+V$ (for a certain class of potentials) as well as the fractional Laplacian $H=(-\Delta)^{\alpha}$, $\alpha>0$. More precisely, we first prove
\begin{lemma}\label{lem2.1.1}
Let $H=-\Delta+V$ be a nonnegative self-adjoint operator on $L^2(\mathbb{R}^n)$. Assume that there exists some $A>0$ such that
\begin{align}\label{equ2.1.3}
\int_{\mathbb{R}^n}{|\hat{V}(\xi)|^2 e^{A|\xi|^2}\,d\xi}=C(A)<\infty.
\end{align}
Then there are positive constants $M_1,\, M_2,\, C,\, c_1$ so that for any $0<r_1\leq M_1$, and $r_2\ge M_2$, one has
\begin{align}\label{equ2.1.4}
\|f\|_2^2\leq Ce^{c_1r_1\cdot r_2}\left(\|\chi(|x|\ge r_1)f\|_2^2+\|\chi(H\ge r^2_2)f\|_2^2\right), \,\,\,\,f\in L^2(\mathbb{R}^n),
\end{align}
\end{lemma}

\begin{proof}
We claim that there exists constant $c_2>0$ such that for any $r_2>0$, the following estimate
\begin{align}\label{equ2.1.5}
\|\chi(H_0\ge r^2_2)f\|_2^2\leq 18\|\chi(H\ge r^2_2/4)f\|_2^2+2e^{-c_2r_2}\|f\|_2^2
\end{align}
holds for every $f\in L^2(\mathbb{R}^n)$. Then combining \eqref{equ2.1.5} with the inequality \eqref{equ2.1.1} yields
\begin{align}\label{equ2.1.6}
\|f\|_2^2\leq Ce^{c_1r_1\cdot r_2}\left(\|\chi(|x|\ge r_1)f\|_2^2+\|\chi(H\ge r^2_2/4)f\|_2^2\right)+2Ce^{(c_1r_1-c_2)\cdot r_2}\|f\|_2^2,
\end{align}
Take $M_1=\frac{c_2}{2c_1}$ and $M_2=\frac{2}{c_2}\ln{2C}$, then for any $0<r_1\leq M_1$, and $r_2\ge M_2$, one has $2Ce^{(c_1r_1-c_2)\cdot r_2}<\frac12$, then the last term in the right hand side of \eqref{equ2.1.6} can be absorbed in the left hand side, which implies the desired estimate \eqref{equ2.1.4}.

In order to prove \eqref{equ2.1.5}, we first point out that we are allowed to use an analytic version of the projection operator $\chi(H\ge a)$ (see \cite{S11}), which plays an important role in our argument below. For any $a>0$, let
\begin{align*}
\chi_a(r)=\begin{cases}
1,\,\, \,\,r\ge a,\\
0,\,\, \,\,r< a;
\end{cases}
\tilde{\chi}_a(r)=\begin{cases}
1,\,\,\,&r\ge a,\\
(\tanh{(r-a/2)}+1)/2,\,\,\,&a/4<r< a,\\
0,\,\,\,&r\le a/4;
\end{cases}
\end{align*}
and $\varphi_a(r)=(\tanh{(r-a/2)}+1)/2$. Clearly, we have
\begin{align}\label{equ2.1.7}
\chi_a(r)\leq \tilde{\chi}_a(r)\leq \chi_{a/4}(r),
\end{align}
and $\varphi_a(r)$ is analytic in the strip $S_{\frac{\pi}{2}}=\{r\in\mathbb{C},\,\, \Im |r|<\frac{\pi}{2}\}$, moreover
\begin{align}\label{equ2.1.8}
|\varphi_a(r)-\tilde{\chi}_a(r)|\leq e^{-\frac{a}{4}},\,\,\,r\in \mathbb{R}.
\end{align}
In view of \eqref{equ2.1.7}, we have (writing $a=r_2^2>0$) the following string of inequalities
\begin{align}\label{equ2.1.9}
\|\chi_a(H_0)f\|&\leq \|\tilde{\chi}_a(H_0)f\|\leq \|\tilde{\chi}_a(H)f\|+\|\left(\tilde{\chi}_a(H_0)-\tilde{\chi}_a(H)\right)f\|\nonumber\\
&\leq \|\chi_{a/4}(H)f\|+\|\left(\tilde{\chi}_a(H_0)-\tilde{\chi}_a(H)\right)\chi(H\ge a/4)f\|\nonumber\\
&+\|\left(\tilde{\chi}_a(H_0)-\tilde{\chi}_a(H)\right)\chi(H<a/4)f\|\nonumber\\
&\leq 3 \|\chi_{a/4}(H)f\|+\|\left(\tilde{\chi}_a(H_0)-\tilde{\chi}_a(H)\right)\chi(H<a/4)f\|,
\end{align}
hence \eqref{equ2.1.5} would follow if one can prove that there exists some $c_2>0$, such that the following estimate
\begin{align}\label{equ2.1.10}
\|\left(\tilde{\chi}_a(H_0)-\tilde{\chi}_a(H)\right)\chi(H<a/4)f\|\leq e^{-\frac{c_2\sqrt{a}}{2}}\|f\|
\end{align}
holds for every $f\in L^2$. Observe that \eqref{equ2.1.8} implies that
\begin{align}\label{equ2.1.11}
\|\varphi_a(H)-\tilde{\chi}_a(H)\|_{L^2-L^2}\leq e^{-a/4},
\end{align}
which decays faster than the bound in \eqref{equ2.1.10}, hence it suffices to prove
\begin{align}\label{equ2.1.12}
\|\varphi_a(H_0)-\varphi_a(H)\chi(H<a/4)f\|\leq e^{-c_2\sqrt{a}}\|f\|.
\end{align}

To prove \eqref{equ2.1.12}, we recall that $\frac{d}{dt}\tanh t=1/\cosh^2 t$ and $\mathcal{F}(1/\cosh t)=1/\cosh t$. Hence we use spectral theorem to write
\begin{align}\label{equ2.1.13}
\left(\varphi_a(H_0)-\varphi_a(H)\right)\chi(H<a/4)f=\frac12\int_{\mathbb{R}}e^{ita/2}{\frac{G(t)}{t}}(e^{itH_{0}}-e^{itH})\chi(H<a/4)f\,dt,
\end{align}
where $G(t)=\frac{1}{\cosh t}\ast\frac{1}{\cosh t}$, therefore one can check that $G(t)$  is analytic in the strip $S_{\frac{1}{4}}=\{ |\Im t|<\frac{1}{4}\}$ and there are constants $C, c>0$ so that $|G(t)|\leq Ce^{-c|\Re t|}$ when $t\in S_{\frac{1}{4}}$. Now we consider the function defined by
\begin{align}\label{equ2.1.14}
T_{f,\,g}(t)=\langle g,\, \frac{G(t)}{t}(e^{itH_{0}}-e^{itH})\chi(H<a/4)f\rangle.
\end{align}
We shall prove that $T_{f,\,g}(t)$ is an analytic function in the strip $S_{\frac{1}{4}}$. First we consider the easier case $t\ne0$, and $t\in S_{\frac{1}{4}}$. We observe that $\chi(H<a/4)f$ is an analytic vector for $e^{itH}$ (see e.g., in \cite{Ne}) in the sense that for any $A>0$,
\begin{align}\label{equ2.1.15}
\sum_{j=0}^{\infty}{\frac{A^j}{j!}\|\frac{d^j}{dt^j}(e^{itH}\chi(H<a/4)f)\|}< \sum_{j=0}^{\infty}{\frac{1}{j!}(\frac{Aa}{4})^j}=e^{\frac{Aa}{4}}<\infty,
\end{align}
which implies that for any $f, g\in L^2$, the function $\langle g, e^{itH}\chi(H<a/4)f\rangle$ has an analytic continuation to the strip $|\Im t|<A$.

To proceed, we note that from our assumption on the potential, we have $V\in C^{\infty}$ and all its derivatives are bounded. We shall further need the following fact concerning the quantitative bound of each derivative for the potential $V$: there exists some $C>0$ such that for all $\alpha=(\alpha_1,\ldots,\alpha_n)\in \mathbb{N}^n$
\begin{align}\label{equ2.1.19}
\|\partial_x^{2\alpha}V(x)\|_{\infty}\leq C^{|\alpha|}\alpha!.
\end{align}
In fact, our assumption \eqref{equ2.1.3} shows that $V$ can be extended to an analytic function on $\mathbb{C}^n$ and in light of  H\"{o}lder's inequality,
\begin{align*}
|\partial_x^{2\alpha}V(x)|&=|(2\pi)^{-\frac{n}{2}}\int_{\mathbb{R}^n}{e^{ix\cdot\xi}(i\xi)^{2\alpha}\hat{V}(\xi)\,d\,\xi}|\\
&\leq (2\pi)^{-\frac{n}{2}}\left(\int|\xi^{4\alpha}|e^{-A|\xi|^2}\,d\,\xi\right)^{\frac12}\left(\int|\hat{V}(\xi)|^2e^{A|\xi|^2}\,d\,\xi\right)^{\frac12}.
\end{align*}
A direct computation yields
\begin{align*}
\int|\xi^{4\alpha}|e^{-A|\xi|^2}\,d\,\xi=\prod_{i=1}^n\int{|\xi_i^{4\alpha_i}|e^{-A\xi_i^2}\,d\xi_i}=
\frac{1}{2^n}\prod_{i=1}^n\frac{\Gamma(2\alpha_i+\frac12)}{A^{2\alpha_i+\frac12}}.
\end{align*}
By stirling's approximation for factorials, we obtain for any $\alpha\in \mathbb{N}^n$
\begin{align*}
\|\partial_x^{2\alpha}V(x)\|_{\infty}&\leq C^{|\alpha|}C(A)^{\frac12}\sqrt{(2\alpha)!}\\
&\leq C^{|\alpha|}\cdot \alpha!,
\end{align*}
which is \eqref{equ2.1.19}. Hence it follows from \cite{JN} that
\begin{align}\label{equ2.1.16}
\|H^{k}_0(1+H)^{-k}\|\leq C^k(k/2)!,\,\,\,k=1,2,\ldots,
\end{align}
which implies that
$$
\sum_{j=0}^{\infty}{\frac{A^j}{j!}\|\frac{d^j}{dt^j}(e^{itH_0}\chi(H<a/4)f)\|}<\infty.
$$
Hence $\chi(H<a/4)f$ is also an analytic vector for $e^{itH_0}$, which shows that $T_{f,\,g}(t)$ is an analytic function when $t\ne0$, and $t\in S_{\frac{1}{4}}$.

Next we shall deal with the case $t=0$, and we rewrite $T_{f,\,g}(t)$ as
\begin{align}\label{equ2.1.17}
T_{f,\,g}(t)=\frac{iG(t)}{t}\int_0^t{\langle g,\,e^{isH_{0}}Ve^{i(t-s)H}\chi(H<a/4)f\rangle\, ds},
\end{align}
and define
\begin{align}\label{equ2.1.18}
\tilde{T}_{f,\,g}(t, s)={\langle g,\,e^{isH_{0}}Ve^{i(t-s)H}\chi(H<a/4)f\rangle}.
\end{align}
It's enough to show that $\tilde{T}_{f,\,g}(t, s)$ is analytic for $t=0$, $s=0$ respectively.

To this end, we note that by an induction argument, it's straightforward to check that for every $j\in \mathbb{Z}^+$,
\begin{align*}
\frac{d^j}{ds^j}\left(e^{isH_0}Ve^{-isH}\chi(H<a/4)f\right)=\sum^j_{k=0}{(-1)^{j-k}\binom j k e^{isH_0}H_0^kVH^{j-k}e^{-isH}\chi(H<a/4)f}.
\end{align*}
Furthermore, by the generalized Leibniz's rule, we have for $1\leq k\leq j$,
\begin{align}\label{equ2.1.20}
H_0^kVH^{j-k}e^{-isH}\chi(H<a/4)f=\sum_{\alpha}{(P^{(\alpha)}(D)V)\cdot\partial^{\alpha}H^{j-k}e^{-isH}\chi(H<a/4)f/\alpha!},
\end{align}
where $P(\xi)=|\xi|^{2k}$, and $\partial_{\xi}^{\alpha}P(\xi)$ is the corresponding symbol for $P^{(\alpha)}(D)$. In view of \eqref{equ2.1.19} and the fact that $\binom {2j} k\leq \binom {2j} j\leq C2^{2j}$ holds for all $1\leq k\leq j$, we have
\begin{align}\label{equ2.1.21}
\|(P^{(\alpha)}(D)V)/\alpha!\|_{\infty}\leq C^{j}\cdot 2^{2j}\cdot\binom{2j}{|\alpha|}\cdot(j-|\alpha|/2)!,\,\,\,\,|\alpha|\leq 2k,
\end{align}
where the constant $C$ doesn't depend on $k, j$. Using \eqref{equ2.1.16}, we also find that
\begin{align}\label{equ2.1.22}
\|\partial^{\alpha}H^{j-k}e^{-isH}\chi(H<a/4)\|_{L^2-L^2}\leq C^{k}\cdot(\frac a 4)^{j}\cdot(|\alpha|/2)!,\,\,\,\,|\alpha|\leq 2k.
\end{align}
Notice that $\sum_{|\alpha|\le 2j}{1}=\binom {2j+n-1} n\le Cj^n$. Therefore it follows from estimates above that
\begin{align}\label{equ2.1.23}
\|\frac{d^j}{ds^j}\left(e^{isH_0}Ve^{-isH}\chi(H<a/4)f\right)\|\leq Cj^n\cdot j!\cdot(Ca)^j.
\end{align}
This indicates that for each fixed $s$, $\tilde{T}_{f,\,g}(t, s)$ is an analytic function of $t$ for $|\Im t|\leq \frac{1}{10Ca}$. Meanwhile,  for each fixed $t$, using the same argument as \eqref{equ2.1.15}, it follows that $\tilde{T}_{f,\,g}(t, s)$ is an analytic function of $s$ for $s\in \mathbb{C}$, this proves the analyticity of $T_{f,\,g}(t)$ at the origin. Hence we have shown that  $T_{f,\,g}(t)$ is indeed an analytic function in the strip $S_{\frac{1}{2}}$. Besides, the arguments above also imply that
\begin{align}\label{equ2.1.24}
|T_{f,\,g}(t)|\leq \frac{C_a}{1+x^2},\,\,\,t=x+iy,\,\,x\in \mathbb{R},\,\,|y|\leq a^{-\frac12}/2,
\end{align}
where the constant $C_a$ is independent of $x\in\mathbb{R}$. Therefore, we obtain that there exists some $c_2>0$ so that
\begin{align}\label{equ2.1.25}
\|\left(\varphi_a(H_0)-\varphi_a(H)\right)\chi(H<a/4)f\|&=\sup_{\|g\|\leq 1}|\int{T_{f,\,g}(t)e^{ita/2}\,dt}|\nonumber\\
&\leq e^{-c_2\sqrt{a}}\|f\|,
\end{align}
which completes the proof of \eqref{equ2.1.12}.
\end{proof}

\begin{remark}\label{rmk.3}
(\romannumeral1) In a recent work of Naki\'{c} et al. \cite{NTTV}, based on Carleman estimates and some interpolation inequalities,  they proved the following inequality for bounded $V$:
\begin{align}\label{equ2.1.255}
\|f\|_2^2\leq C\exp\{cr_1\cdot r_2+r_1^{4/3}\|V\|_{\infty}\}\left(\|\chi(|x|\ge r_1)f\|_2^2+\|\chi(H\ge r^2_2)f\|_2^2\right), \,\,\,\,f\in L^2(\mathbb{R}^n).
\end{align}

(\romannumeral2) We discuss another possible approach by using spectral measures. Let $E'_H(\lambda)$ and $E'_0(\lambda)$ denote the spectral measure of $H$ and $H_0$ respectively. If one assumes that there exists a continuous function $g(\lambda)\ge c_0>0$ for all $\lambda>0$ such that
\begin{align}\label{equ2.1.26}
E'_H(\lambda)=g(\lambda)E'_0(\lambda)\,\,\,\,\lambda>0.
\end{align}
Then there is some constant $C>0$ so that the following estimate
\begin{align}\label{equ2.1.27}
\|\chi(H_0\ge r)f\|_2^2\leq C \|\chi(H\ge r)f\|_2^2,\,\,\,\, r>0
\end{align}
holds for any $f\in L^2$, which is much stronger than \eqref{equ2.1.5}. Indeed, by the spectral theorem and the assumption on $g$ mentioned above, we obtain
\begin{align}
\|\chi(H_0\ge r)f\|_2^2 &=\int_{\lambda\ge r}{(E'_0(\lambda)f, f)\,d\lambda}\nonumber\\
&\leq c_0^{-1}\int_{\lambda\ge r}{(E'_H(\lambda)f, f)d\lambda}\nonumber\\
&=c_0^{-1} \|\chi(H\ge r)f\|_2^2.\nonumber
\end{align}
\end{remark}


At the end of this subsection, we consider the fractional Laplacian $H=(-\Delta)^{\alpha}$, $\alpha>0$, we have the following
\begin{lemma}\label{cor2.5}
Let $n\ge 1$,  $H=(-\Delta)^{\frac{s}{2}}$, $s>0$. Then there exists a constant $C(n, r_1, r_2)$ satisfying \eqref{equ2.1.2} for any $r_1, r_2>0$,  such that
\begin{align}\label{equ2.1.28}
\|f\|^2\leq C(n, r_1, r_2)\left(\|\chi(|x|\ge r_1)f\|^2+\|\chi(H\ge r_2^{s})f\|^2\right),\,\,\, f\in  L^2(\mathbb{R}^n).
\end{align}
\end{lemma}
\begin{proof}
Observe that by the Plancherel theorem, we have
\begin{align*}
\|\chi(H\ge r_2^{s})f\|^2=\int_{|\xi|\ge r_2}|\hat{f}(\xi)|^2\,d\xi=\|\chi(-\Delta\ge r_2^{2})f\|^2.
\end{align*}
Hence the desired estimate \eqref{equ2.1.28} follows directly from the Nazarov Uncertainty principle \eqref{equ2.1.1}-\eqref{equ2.1.2}.
\end{proof}

\subsection{Minimal escape velocities}\label{sec2.2}
In this subsection, we first collect some known minimal velocity estimates for the unitary evolutions $e^{-itH}$. Then we discuss examples of operators which these estimates apply to. As already pointed out in the introduction, results established in this subsection play an essential role in our proof of the observability estimates in Sect. \ref{sec3}.

The starting point is Mourre's inequality \cite{Mo}, whose fundamental idea is to find observable (self-adjoint operator) $A$ such that the commutator $i[H,\, A]$ is conditionally positive, in the sense that
\begin{align}\label{equ2.15}
E_{\Delta}i[H,\, A]E_{\Delta}\ge \theta E_{\Delta},\,\,\, \theta>0
\end{align}
for some compact interval $\Delta\subset\mathbb{R}$, where $E_{\Delta}$ is the corresponding spectral projection of $H$.  To provide further intuition  in the understanding of condition \eqref{equ2.15}, let us consider $H=-\frac12\Delta$, and $A$ is the generator of dilation:
\begin{align}\label{equ2.150}
A=\frac12(x\cdot p+p\cdot x), \,\,\,\, i[H,\, A]=2H, \,\, p=-\imath\nabla_x.
\end{align}
Observe that $A=i[H,\, \frac{x^2}{2}]$, then \eqref{equ2.15} can be written as $\partial_t^2\langle x^2\rangle_t\ge 2\theta$, where $\langle x^2\rangle_t=\langle \psi_t,\, x^2\psi_t\rangle$, and $\psi_t=e^{-itH}\psi\in E_{\Delta}$, which in turn implies that
\begin{align}\label{equ2.15'}
\langle x^2\rangle_t\ge \theta t^2+O(t),\,\,\, t\rightarrow \infty.
\end{align}
\par
The second key ingredient is that the multiple commutator of $H$ and $A$ are well behaved. More precisely, we assume that for any $g\in C_0^{\infty}(\mathbb{R})$
\begin{align}\label{equ2.16}
\|ad_{A}^{(k)}g(H)\|\leq C_k,\,\,\, k=1,2.
\end{align}
We regard \eqref{equ2.16} as a regularity assumption and refer to the monograph \cite{AMG} for extensive discussion on this. We mention here that the commutator method, used in the proof of minimal velocity estimates, can be viewed as an abstract version of the integration by parts arguments. Hence for higher value of $k$ that \eqref{equ2.16} is satisfied, the faster decay (for $t$) is expected. Here, we only assume that $k\le 2$, which is good enough for our applications.
\par
Having discussed the main assumptions, we now present the following type of minimal velocity estimates.

\begin{lemma}\cite[Theorem 5.2]{SS}\label{lem2.3}
Assume $H$ and $A$ satisfy \eqref{equ2.15} and \eqref{equ2.16}. Furthermore, if
\begin{align}\label{equ2.17}
\|(1+|A|^2)^{\alpha/2}(H+i)^{-1}(1+|x|^2)^{-\alpha/2}\|\leq C\,\,\,\, \text{for}\,\,\,  0\leq\alpha\leq 1.
\end{align}
Then for any $v<v_{min}=\sqrt{\theta}$, and $0<m<1$
\begin{align}\label{equ2.18}
\|F(\frac{|x|}{t}<v)e^{-itH}\psi\|\leq C(1+|t|)^{-m}\left(\|\psi\|+\||A|\psi\|\right),
\end{align}
where $\psi=E_{\Delta}\psi\in D(|A|)$.
\end{lemma}

A few remarks are given in order. First, the condition \eqref{equ2.17} is not hard to verify in applications, see, e.g. \cite{PSS} for the case of Schr\"{o}dinger operators. Next, the estimate \eqref{equ2.18}, can be thought as a quantitative version of the estimate \eqref{equ2.15'},  shows  that if the initial data is localized in the sense that $\psi=E_{\Delta}\psi\in D(|A|$), then the support of the distribution $|e^{-itH}\psi|^2$ is asymptotically contained in the region $|x|\ge t\sqrt{\theta}$, as $t\rightarrow \infty$, up to a remainder of order $t^{-m}$, with any $m<1$.

We proceed with another type of minimal velocity estimates. Before stating the result, we briefly illustrate that the main idea is to decompose the state into outgoing and incoming waves by means of the spectral decomposition of $A$. Such idea was introduced by Enss \cite{En} in order to prove asymptotic completeness, More precisely, we say that a state $\psi$ is outgoing/incoming if $\psi\in \text{Ran} P^{\pm}(A)$, where $P^{\pm}$ denotes the projection on $\mathbb{R}^{\pm}$, see e.g. \cite{Mo82,Jen,RT}. Roughly speaking, the advantage of this decomposition is that outgoing components will evolve towards spatial infinity (and never come back) as $t\rightarrow\infty$, whereas incoming parts will evolve towards spatial infinity as $t\rightarrow-\infty$. In particular, we have the following

\begin{lemma}\cite[Theorem 1.1, 1.2]{HSS}\label{lem2.4}
 Assume $H$ and $A$ satisfy \eqref{equ2.15} and \eqref{equ2.16}. Let $\chi^{\pm}$ be the characteristic function of $\mathbb{R}^{\pm}$. Then for $0<m<1$,
\begin{align}\label{equ2.20}
\|\chi^{-}(A-a-vt)e^{-itH}g(H)\chi^{+}(A-a)\|\leq C(1+|t|)^{-m}
\end{align}
holds for any $g\in C_0^{\infty}(\Delta)$, any $0<v<\sqrt{\theta}$, uniformly in $a\in \mathbb{R}$. In particular, if $H=-\Delta+V$, and $A=\frac12(x\cdot p+p\cdot x)$ satisfy the assumption above,  then
\begin{align}\label{equ2.21}
\|\chi^{-}(|x|^2-2at-vt^2)e^{-itH}g(H)\chi^{+}(A-a)\|\leq C(1+|t|)^{-m}.
\end{align}
\end{lemma}

\begin{remark}\label{rmk2.8}
(\romannumeral1) We note that the estimates above are uniform with respect to $a\in \mathbb{R}$. Note that when $a<0$, thus the state $\chi^{+}(A-a)\psi$ contains incoming component. However, the estimate \eqref{equ2.21} indicates that after finite time ($\approx\frac{-a}{v}$), the incoming part turns out to be outgoing.\\
(\romannumeral2) In our applications, we shall further investigate the behavior of the constants in \eqref{equ2.21} and \eqref{equ2.21} when $g$ varies in a suitable way. In particular, we shall prove that (see Corollary) it's also uniform when $g$ is replaced by $g_k(\cdot)=g(\frac{\cdot}{2^k})$, $k=1,2,\ldots$.
\end{remark}

Now we turn to concrete examples. First we consider Schr\"{o}dinger operators $H=-\frac12\Delta+V$. We note that in the next section, we shall work with $H_R=-\frac12\Delta+R^2V(R\cdot)$, $R>0$ via a scaling argument. Hence we make the following assumption on $V_R=R^2V(R\cdot)$. Assume that there are constants $a_k, b_k>0$ ($k=0,1,2$) independent with $R>0$ and $0<a_0<1$ such that
\begin{align}\label{equ2.23}
\begin{cases}
\|(x\cdot\nabla)^k V_Rf\|\leq \frac{a_k}{2}\|\Delta f\|+b_k\|f\|,\,\,\, \text{for any}\,\,R>0,\,\,k=0, 1, 2.\\
-(x\cdot\nabla)V_R\ge 0,\,\,\, \text{for any}\,\,R>0.
\end{cases}
\end{align}

Under this assumption, we have
\begin{corollary}\label{cor2.9}
Let $V_R=R^2V(R\cdot)$ satisfy the condition \eqref{equ2.23} above. Then there exists a constant $C$ uniformly in $R>0, k\ge 1$ and $a\in \mathbb{R}$ such that for $g_k(\cdot)=g(\frac{\cdot}{2^k})$, $k=1,2,\ldots$
\begin{align}\label{equ2.24}
\|\chi^{-}(|x|^2-2at-vt^2)e^{-itH_R}g_k(H_R)\chi^{+}(A-a)\|\leq C(1+|t|)^{-m}.
\end{align}
\end{corollary}
\begin{proof}
We first note that
\begin{align}\label{equ2.25}
i^kad_A^{(l)}(H_R)=-2^{l-1}\Delta+(-x\cdot\nabla)^l V_R,
\end{align}
then it's easy to check that the Mourre's inequality \eqref{equ2.15} is satisfied. In the following, we claim that
\begin{align}\label{equ2.26}
\|ad_A^{(l)}g_k(H_R)\|\leq C, \,\,\text{for}\,\, l=1,2,
\end{align}
where the constant  can be chosen independent with  $R>0$ and  $k=1,2,\ldots$. To show \eqref{equ2.26}, we will make use of the Helffer-Sj\"{o}strand formula (see e.g. \cite[p.24]{Da})
$$
g_k(H_R)=-\frac{1}{\pi}\int_{\mathbb{C}}{\frac{\partial\tilde{g_k}(z)}{\partial\bar{z}}(H_R-z)^{-1}\,L(dz)},
$$
which implies
\begin{align}\label{equ2.27}
[A,\, g_k(H_R)]=-\frac{1}{\pi}\int_{\mathbb{C}}{\frac{\partial\tilde{g_k}(z)}{\partial\bar{z}}[A,\, (H_R-z)^{-1}]\,L(dz)},
\end{align}
where $L(dz)$ denotes the Lebesgue measure on $\mathbb{C}$ and $\tilde{g_k}\in C_0^{\infty}(\mathbb{C})$ is an almost analytic continuation of $g_k$ supported in a small neighborhood of $\text{supp}\,g_k\subset 2^k\Delta$. More explicitly, one can take
$$
\tilde{g_k}(z)=\sum_{r=0}^{N}\frac{g_k^{(r)}(x)(iy)^r}{r!}\cdot \tau(\frac{y}{\langle x\rangle}),\,\,\,\,\, z=x+iy,
$$
where $\tau\in C_0^{\infty}(\mathbb{R})$, $\tau(s)=1$, if $|s|<1$, and  $\tau(s)=0$, if $|s|>2$. Note that
$$
|g_k^{(r)}(x)|\leq C,
$$
which is uniform with $k=1,2,\ldots$. It then follows that
\begin{align}\label{equ2.29}
|\partial\tilde{g_k}(z)|\leq C|\Im z|^N.
\end{align}
Using \eqref{equ2.25} and our assumption on $V_R$, we obtain that
\begin{align}\label{equ2.30}
\|[A,\, (H_R-z)^{-1}]\|&=\|(H_R-z)^{-1}[A,\, H_R](H_R-z)^{-1}\| \nonumber\\
&\leq C|\Im z|^{-1}.
\end{align}
Thus the claim follows by combining \eqref{equ2.29} and \eqref{equ2.30}. Therefore the commutator condition $\eqref{equ2.16}$ is also verified with a uniform upper bound. Now the conclusion is followed by a step-by-step repetition of the proof in \cite{HSS}.
\end{proof}

Next, we consider applications to the fractional Laplacian. Let $H=(-\Delta)^{\frac{s}{2}}$ with domain the Sobolev space $H^{s}(\mathbb{R}^n)$. It's well-known that $\sigma(H)=\sigma_c(H)=[0,\,\infty)$. We have the following
\begin{corollary}\label{cor2.10}
For $s\ge 1$, then estimates \eqref{equ2.18} and \eqref{equ2.24} are valid with $H=H_R=(-\Delta)^{\frac{s}{2}}$.
\end{corollary}
\begin{proof}
We first observe that
\begin{align}\label{equ2.31}
i[H, A]=sH,\,\,\,\text{where}\,\,\, A=\frac12(x\cdot p+p\cdot x).
\end{align}
For $b>a>0$, it then follows that
\begin{align*}
E_{[a, b]}i[H,\, A]E_{[a, b]}\ge as E_{[a, b]}
\end{align*}
Hence the Mourre's inequality \eqref{equ2.15} is satisfied.

We proceed to verify the commutator estimates \eqref{equ2.16} with $g$ replaced by $g_k(\cdot)=g(\frac{\cdot}{2^k})$, $k=1,2,\ldots$. Similar to the proof in Corollary \ref{cor2.9}, it suffices to check the following
\begin{align}\label{equ2.32}
\|[A,\, (H_R-z)^{-1}]\|\leq C|\Im z|^{-1},\,\,\,\text{for any}\,\,\, R>0,
\end{align}
which in turn follows by observing (note that $H=H_R=(-\Delta)^{\frac{s}{2}}$)
$$
[A,\, (H-z)^{-1}]=is(H-z)^{-1}H(H-z)^{-1}.
$$

We further point out that the estimate  \eqref{equ2.17} is also satisfied for $(-\Delta)^{\frac{s}{2}}$, $s\ge 1$. Indeed, by interpolation, it suffices to prove $\alpha=1$. This follows by using the fact
\begin{align}\label{equ2.33}
\||p|(H+i)^{-1}\|\leq C,\,\,\,\,\,\,s\ge 1
\end{align}
and writing
$$
p\cdot x(H+i)^{-1}(1+|x|^2)^{-\frac12}=S_1+S_2,
$$
where
$$
S_1=\left(p(H+i)^{-1}\right)\cdot\left(x(1+|x|^2)^{-\frac12}\right),
$$
and
$$
S_2=\left(-isp(H+i)^{-1}\right)\cdot\left(p(-\Delta)^{\frac s2-1}(H+i)^{-1}(1+|x|^2)^{-\frac12}\right).
$$
Having checked all the needed conditions, the result then follows by applying Lemma \ref{lem2.3}, Lemma \ref{lem2.4} and combining the proof in Corollary \ref{cor2.9}.
\end{proof}

\section{Sharp observability inequalities for Schr\"{o}dinger type equations}\label{sec3}
\subsection{Observability inequalities}\label{sec3.1}
In this subsection, we shall show how tools established in section \ref{sec2} could be used to recover the initial data by observing the solutions at two different points in time, and each time outside of a ball. More precisely, For Schr\"{o}dinger equations with potentials, we shall prove the following
\begin{theorem}\label{thm3.1}
Let $n\ge 3$, $H=-\Delta+V$ and assume that $V$ satisfies \eqref{equ2.6} and \eqref{equ2.23}. Let $u(x, t)$ be the solution of the following Cauchy problem
\begin{equation}\label{equ3.1}
i\partial_{t}u(x, t) =H u(x, t), \qquad u(x, 0)=u_{0}(x)\in L^2.
\end{equation}
Then there exists some $\sigma>0$, such that for any $r>0$, $t_2>t_1\ge 0$ with $t_2-t_1>r^2$, we have
\begin{align}\label{equ3.2}
\|u_0\|^2\leq C\left(\int_{|x|\ge \frac{r}{\sqrt{T_0}}}{|u(x,t_1)|^2\,dx}+\int_{|x|\ge \frac{\sqrt{T_0}\sigma\cdot(t_2-t_1)}{r}}{|u(x,t_2)|^2\,dx}\right),
\end{align}
where $T_0$  is some large but fixed number and the constant $C$ depends only on the dimension.
\end{theorem}
\begin{proof}
We first point out that the repulsive conditions in \eqref{equ2.23} ensure that $H=-\Delta+V$ has no positive eigenvalues (see e.g. \cite[Theorem \uppercase\expandafter{\romannumeral13}. 60]{RS}). Since $e^{-itH}$ is unitary, it follows that \eqref{equ3.2} is equivalent to the case $t_1=0$:
\begin{align}\label{equ3.3}
\|u_0\|^2\leq C\left(\int_{|x|\ge \frac{r}{\sqrt{T_0}}}{|u_0|^2\,dx}+\int_{|x|\ge \frac{\sqrt{T_0}\sigma t}{r}}{|u(x,t)|^2\,dx}\right),\,\,\, t>r^2.
\end{align}
In order to prove \eqref{equ3.3}, we note that by using the scaling identity $U_R^{-1}H_RU_R=R^2H$ in \eqref{equ2.8} with $R=\frac{r}{\sqrt{T_0}}>0$, where $T_0$ is some fixed number determined later, it suffices to prove that for $H_R=-\frac12\Delta+V_R$, there exists a uniform constant $C>0$ such that for any $u_0\in L^2$,
\begin{align}\label{equ3.4}
\|u_0\|^2\leq C\left(\int_{|x|\ge 1}{|u_0|^2\,dx}+\int_{|x|\ge \sigma t}{|e^{-itH_R}u_0|^2\,dx}\right),\,\,\, t>T_0.
\end{align}
We proceed to observe that \eqref{equ3.4} can be easily deduced from the following
\begin{align}\label{equ3.5}
\|u_0\|^2\leq C_1\int_{|x|\ge \sigma t}{|e^{-itH_R}u_0|^2\,dx},\,\,\, t>T_0,\,\,\,\text{supp}\,u_0\subset B(0, 1).
\end{align}
Indeed, assuming that \eqref{equ3.5} is true, then for any $u_0\in L^2$, we write $u_0=u_{01}+u_{02}$, where $u_{01}=\chi(|x|\le 1)u_0$, then applying \eqref{equ3.5} to $u_{01}$ and using Minkowski inequality we obtain \eqref{equ3.4} with $C=C_1+1$. The advantage of this reduction is that it allows us to choose initial data localized in the unit ball.

Now we apply the uncertainty principle associated with $H_R$ established in Lemma \ref{lem2.2}. More precisely, if $\text{supp}\,u_0\subset B(0, 1)$, then it follows from \eqref{equ2.9} that there exists some fixed $\delta_1>0$ such that
\begin{align}\label{equ3.6}
\|u_0\|^2\leq C\|\chi(H_R\geq \delta_1)u_0\|^2,\,\,\,\, \text{for any}\,\, R>0.
\end{align}
Thus \eqref{equ3.5} would be followed if we can prove that there exists a uniform constant $\sigma$ doesn't depend on $R$ and $t$, such that
\begin{align}\label{equ3.7}
\|\chi(H_R\geq \delta_1)u_0\|^2\leq C_2\int_{|x|\ge \sigma t}{|e^{-itH_R}u_0|^2\,dx}+\varepsilon\|u_0\|^2,\,\,\,\, t>T_0,
\end{align}
with $C_1\varepsilon<1$.

The role of the uncertainty principle is to make sure that the energy of the initial data has a positive lower bound, which provides the possibility to use the method of  minimal escape velocities. In order to use tools from section \ref{sec2.2}, we break the initial data into a sum of finite energy and write
$$
\chi(H_R\geq \delta_1)u_0=\chi(\delta_1\leq H_R\leq N )u_0+\sum_{k=n_0}^{\infty}\varphi(\frac{H}{2^k})u_0,
$$
where $N=2^{n_0}$ is some large fixed number and $\varphi$ stands for some well-chosen smoothed characteristic function. Hence \eqref{equ3.7} is valid if we can prove
\begin{align}\label{equ3.8}
\|\chi(\delta_1\leq H_R\leq N )u_0\|^2\leq C\|\chi(\delta_1\leq H_R\leq N )\chi(|x|\ge \sigma t)e^{-itH_R}u_0\|^2+\frac{\varepsilon}{2}\|u_0\|^2,\,\,\,\, t>T_0,
\end{align}
and
\begin{align}\label{equ3.9}
\|\varphi(H/2^k)u_0\|^2\leq C\|\varphi(H/2^k)\chi(|x|\ge \sigma t)e^{-itH_R}u_0\|^2+\varepsilon2^{-k/4}\|u_0\|^2,\,\,\,\, t>T_0.
\end{align}

We first investigate \eqref{equ3.8}. Notice that it follows from Corollary \ref{cor2.9} that for any fixed $\sigma<\sqrt{\delta_1}$ and $0<m<1$,
\begin{align}\label{equ3.10}
\|\chi(|x|\le \sigma t)e^{-itH_R}\chi(\delta_1\leq H_R\leq N )u_0\|^2\leq \frac{C}{\langle t\rangle^{2m}}\|\chi(\delta_1\leq H_R\leq N )u_0\|^2,
\end{align}
then choose some fixed $T_0$ large enough, and for $t\ge T_0$, \eqref{equ3.10} implies that
\begin{align}\label{equ3.11}
\|\chi(\delta_1\leq H_R\leq N )u_0\|^2\leq C\|\chi(|x|\ge \sigma t)e^{-itH_R}\chi(\delta_1\leq H_R\leq N)u_0\|^2,\,\,\,\, t>T_0.
\end{align}
Compared to the desired form \eqref{equ3.8}, we must commute the factor $\chi(\delta_1\leq H_R\leq N)$ to the left of the term $|\chi(|x|\ge \sigma t)$. To this end, we now apply Lemma \ref{lemA1} with $A=H_R\chi(H_R)$, $B=\chi(|x|\ge \sigma t)$ and note that
$$
\|[H_R\chi(H_R), \chi(|x|\ge \sigma t)]\|\leq C,
$$
where the constant $C$ is uniform with respect to $R$ and $t$. Thus we have
\begin{align*}
\|\chi(|x|\ge \sigma t)e^{-itH_R}\chi(\delta_1\leq H_R\leq N)u_0\|^2&\leq \|\chi(\delta_1\leq H_R\leq N)\chi(|x|\ge \sigma t)e^{-itH_R}u_0\|^2\\
&+\|[\chi(\delta_1\leq H_R\leq N), \chi(|x|\ge \sigma t)]e^{-itH_R}u_0\|^2\\
&\leq \|\chi(\delta_1\leq H_R\leq N )\chi(|x|\ge \sigma t)e^{-itH_R}u_0\|^2+\frac{\varepsilon}{2}\|u_0\|^2,
\end{align*}
provided $t>T_0$, which proves \eqref{equ3.8}.

We are left to prove \eqref{equ3.9}. We set $g_k=\varphi(H_R/2^k)u_0$, notice that $\text{supp}\, u_0\subset B(0,1)$ and $H_R\sim 2^k$. Furthermore, under our assumption  \eqref{equ2.6} and \eqref{equ2.23} on $V$, we have that $|p|\sim2^{k/2}$, thus classically, in phase space, we have $A\gtrsim -2^{k/2}$, without loss of generality, we write
$$
g_k=\chi^{+}(A+2^{k/2})g_k.
$$
Then apply \eqref{equ2.24} in Corollary \ref{cor2.9} with $a=-2^{k/2}$, we find
\begin{align}\label{equ3.12}
\|\chi^{-}(|x|^2+2^{k/2+1}t-v_kt^2)e^{-itH_R}\varphi(H_R/2^{k})\chi^{+}(A+2^{k/2})\|\leq C(1+|t|)^{-m},\,\,\,v_k\approx 2^{k/2},
\end{align}
which implies, after choosing sufficiently large $t$, that
\begin{align*}
\|\varphi(H/2^k)u_0\|^2&\leq \|\chi^{+}(|x|^2+2^{k/2+1}t-v_kt^2)e^{-itH_R}g_k(H_R)\chi^{+}(A+2^{k/2})u_0\|^2\\
&\leq \|\chi(|x|^2 \ge \sigma^2t^2 )\chi(|x|\ge \sigma t)e^{-itH_R}u_0\|^2+\frac{\varepsilon}{2}\|u_0\|^2,
\end{align*}
where in the last inequality, we have used the simple fact that
$$
v_kt^2-2^{k/2+1}t\ge \sigma^2t^2,\,\,\,\,k\ge n_0.
$$
Then we apply Lemma \ref{lemA1} with $A=H_R\varphi(H_R/2^k)$, $B=\chi(|x|\ge \sigma t)$ and note that
$$
\|[H_R\varphi(H_R/2^k), \chi(|x|\ge \sigma t)]\|\leq Ct^{-1}2^{-\frac k4},
$$
which indicates \eqref{equ3.9} and the proof is complete.
\end{proof}

We mention that the proof in Theorem \ref{thm3.1} can be applied to the fractional Schr\"{o}dinger equations by using Lemma \ref{cor2.5} and Corollary \ref{cor2.10}.
\begin{theorem}\label{thm3.2}
Assume $n\ge 1$, and let $u(x,t)$ be the solution of the following Cauchy problem
\begin{equation}\label{equ3.1}
i\partial_{t}u(x,t) =(-\Delta)^{\frac{s}{2}} u(x,t), \qquad u(x, 0)=u_{0}(x)\in L^2,\,\,\,\, s\ge 1.
\end{equation}
Then there exist some $\sigma>0$ and some large but fixed number  $T_0>0$, such that for any $r>0$, $t_2>t_1\ge 0$ with $t_2-t_1>r^{s}T_0$, we have
\begin{align}\label{equ3.14}
\|u_0\|^2\leq C\left(\int_{|x|\ge r}{|u(x,t_1)|^2\,dx}+\int_{|x|\ge\frac{\sigma(t_2-t_1)}{r^{s-1}}}{|u(x,t_2)|^2\,dx}\right).
\end{align}
\end{theorem}
\begin{remark}\label{rmk3.1}

The above proof extends to the case where for a general $H$, one can construct $\tilde A$
satisfying $i [H,\tilde A]=H$, as well as regularity assumptions as before.
Furthermore, we require that the principal symbol of $\tilde A$ will be the same as that of $A$.
Such $\tilde A$ were constructed for large class of potentials, without the repulsive assumption on $V$.
See \cite{GLS,LS}.
\end{remark}

\subsection{Sharpness of the observability inequalities}\label{sec3.2}
The purpose of this subsection is to show the optimality of the inequalities established in section \ref{sec3.1}.  We recall that it was observed in \cite{WWZ} that for the free Schr\"{o}dinger equation, the observability inequality can't be replaced by
\begin{align}\label{equ3.15}
\int_{\mathbb{R}^n}|u_0|^2\,dx\leq C\left(\int_{|x|\ge r_1}|e^{it_1\Delta}u_0|^2\,dx+\int_{|x|\le r_2}|e^{it_2\Delta}u_0|^2\,dx\right), \,\,u_0\in L^2(\mathbb{R}^n)
\end{align}
for any fixed $r_1, r_2>0$ and $t_2>t_1\ge 0$. In other words, we can't expect to recover the solution by observing it at two different points in time, one point outside a ball while the other inside a ball with any fixed radius. However, since the argument in \cite{WWZ} again relies heavily on the representation formula \eqref{equ1.3} for the solution $e^{it\Delta}u_0$, thus it doesn't apply to other situations. We shall point out that by using {\bf minimal velocity estimates}, one can treat more general cases.

\begin{theorem}\label{thm3.3}
Let $H=-\Delta+V$ satisfy the assumption in Theorem \ref{thm3.1}. Then one can find a sequence of $L^2$ functions $\{f_k\}_{k\in\mathbb{Z}}$ with
\begin{align}\label{equ3.16}
\int_{\mathbb{R}^n}|f_k|^2\,dx=1,
\end{align}
moreover, there exist some $\sigma>0$ and some large enough $T>0$, such that for any $t>T$ and any fixed $r_1>0$,
\begin{align}\label{equ3.17}
\lim_{k\rightarrow \infty}\int_{|x|\ge r_1}|f_k|^2\,dx=\lim_{k\rightarrow \infty}\int_{|x|\le \sigma t}|e^{itH}f_k|^2\,dx=0.
\end{align}
\end{theorem}

\begin{proof}
Choose $f\in C_0^{\infty}(\mathbb{R}^n)$ such that $\|f\|_{L^2}=1$ and $\text{supp}\, f\subset B(0, 1)$. Then we set $f_k=U_kf=k^{\frac n2}f(k\cdot)$, $k=1,2,\cdots$, where $U_k$ is the scaling operator in \eqref{equ2.8}. Since $U_k$ is an isometry on $L^2(\mathbb{R}^n)$, \eqref{equ3.16} follows immediately. Moreover, a scaling argument shows that for any fixed $r_1>0$,
\begin{align}\label{equ3.18}
\lim_{k\rightarrow \infty}\int_{|x|\ge r_1}|f_k|^2\,dx=\lim_{k\rightarrow \infty}\int_{|x|\ge kr_1}|f|^2\,dx=0.
\end{align}
Hence it suffices to prove the $L^2$ norm $\|\chi(|x|\le \sigma t)e^{itH}f_k\|$ goes to zero as $k\rightarrow \infty$. To this end, we write
\begin{align}\label{equ3.19}
\|\chi(|x|\le \sigma t)e^{itH}f_k\|\leq \|\chi(|x|\le \sigma t)e^{itH}\chi(H\le 1)f_k\|+\|\chi(|x|\le \sigma t)e^{itH}\chi(H\ge 1)f_k\|.
\end{align}
On one hand, it follows from \eqref{equ2.9}, \eqref{equ2.10'} and the fact $f_k=\chi(|x|\le \frac1k)f_k$ that
\begin{align}\label{equ3.20}
\|\chi(H\le 1)\chi(|x|\le \frac1k)f_k\|\leq C_H\cdot k^{-\frac12}\rightarrow 0,\,\,\,\text{as}\,\, k\rightarrow \infty.
\end{align}
On the other hand, we observe that
\begin{align}\label{equ3.21}
\|\chi(|x|\le \sigma t)e^{itH}\chi(H\ge 1)f_k\|=\|\chi(|x|\le \frac{\sigma \tilde{t}}{k})e^{i \tilde{t}H_R}\chi(H\ge \frac {1}{k^2})f\|,\,\,\,\,
\end{align}
where $\tilde{t}=k^2t,\,\, R=\frac1k$. Then by Lemma \ref{lem2.3}, we can choose  $\sigma<1$ small enough  and a uniform constant $C$ such that
\begin{align}\label{equ3.22}
\|\chi(|x|\le \frac{\sigma \tilde{t}}{k})e^{i \tilde{t}H_R}\chi(\frac {1}{k^2}\leq H\leq N)f\|\leq C(1+|\tilde{t}|)^{-m},\,\,\,\,
\end{align}
where $N=2^{n_0}$ for some $n_0\in \mathbb{N}$. Also
\begin{align}\label{equ3.23}
\|\chi(|x|\le \frac{\sigma \tilde{t}}{k})e^{i \tilde{t}H_R}\chi(2^l\leq H\leq 2^{l+1})f\|\leq C(2^l|\tilde{t}|)^{-m},\,\,\,\,l=n_0, n_0+1,\cdots.
\end{align}
Combining \eqref{equ3.22} and \eqref{equ3.23}
\begin{align}\label{equ3.24}
\|\chi(|x|\le \frac{\sigma \tilde{t}}{k})e^{i \tilde{t}H_R}\chi(H\ge \frac {1}{k^2})f\|\leq C|\tilde{t}|^{-m},\,\,\,\,t>T.
\end{align}
Therefore it follows from \eqref{equ3.24}, \eqref{equ3.19}, \eqref{equ3.20} and \eqref{equ3.21} that
\begin{align}\label{equ3.25}
\lim_{k\rightarrow \infty}\|\chi(|x|\le \sigma t)e^{i tH}f_k\|=0,
\end{align}
which completes the proof.
\end{proof}

In the special case where $H=(-\Delta)^{\frac s2}$ ($s\ge 1$), though the solution can't be written as the explicit form like \eqref{equ1.3} when $s\ne2$, we can still write the solution $e^{it(-\Delta)^{\frac s2}}f$ as a oscillatory integral and prove by a integration by parts argument. More precisely, we have

\begin{theorem}\label{thm3.4}
Let  $H=(-\Delta)^{\frac s2}$, $s\ge 1$. Then one can find a sequence of $L^2$ functions $\{f_k\}_{k\in\mathbb{Z}}$ with
\begin{align}\label{equ3.26}
\int_{\mathbb{R}^n}|f_k|^2\,dx=1,
\end{align}
and there exists some $\sigma>0$, such that for any $t>0$ and any fixed $r_1>0$,
\begin{align}\label{equ3.27}
\lim_{k\rightarrow \infty}\int_{|x|\ge r_1}|f_k|^2\,dx=\lim_{k\rightarrow \infty}\int_{|x|\le \sigma t}|e^{it(-\Delta)^{\frac s2}}f_k|^2\,dx=0.
\end{align}
\end{theorem}

\begin{proof}
Let $f_k$ be as in Theorem \ref{thm3.3}. By \eqref{equ3.19}, it's enough to prove that
\begin{align}\label{equ3.28}
\lim_{k\rightarrow \infty}\|\chi(|x|\le \sigma t)e^{itH}\chi(H\ge 1)f_k\|=0.
\end{align}
 We first observe that by scaling and the homogeneity of $(-\Delta)^{\frac s2}$
\begin{align}\label{equ3.29}
\|\chi(|x|\le \sigma t)e^{itH}\chi(H\ge 1)f_k\|=\|\chi(|x|\le \frac{\sigma \tilde{t}}{k^{s-1}})e^{i \tilde{t}H}\chi(H\ge \frac {1}{k^s})f\|,\,\,\,\,\tilde{t}=t\cdot k^s.
\end{align}
Then we write
\begin{align}\label{equ3.30}
e^{i \tilde{t}H}\chi(H\ge \frac{1}{k^s})f=\int_{|\xi|\ge\frac1k}{e^{-i( \tilde{t}|\xi|^s-x\cdot\xi)}\hat{f}(\xi)\,d\xi}.
\end{align}
To proceed, we notice that in the region $|x|\le \frac{\sigma \tilde{t}}{k^{s-1}}$ with $\sigma\le \frac12$, a direct computation shows that there exists a uniform constant $C$ such that
\begin{align}\label{equ3.31}
|\nabla_{\xi}\frac{(\tilde{t}|\xi|^s-x\cdot\xi)}{|x|+|\tilde{t}|}|\ge Ck^{1-s}.
\end{align}
Then we obtain after taking integration by parts
\begin{align}\label{equ3.32}
\|\chi(|x|\le \frac{\sigma \tilde{t}}{k^{s-1}})e^{i \tilde{t}H}\chi(H\ge \frac {1}{k^s})f\|\leq C_N|tk|^{-N}\rightarrow 0,\,\,\,\text{as}\,\,k\rightarrow \infty,
\end{align}
which implies \eqref{equ3.28} by \eqref{equ3.29}, hence completes the proof.
\end{proof}

\section{Applications}\label{sec4}
Now we turn to the application. First we mention that the observability inequalities established in Sect. \ref{sec3.1} may also be regarded as a kind of quantitative unique continuation property for the corresponding solutions. In particular, we consider
\begin{equation}\label{equ4.1}
i\partial_{t}u(x, t) =H u(x, t), \qquad u(x, 0)=u_{0}(x)\in L^2(\mathbb{R}^n),\,\,\,\,n\ge 3.
\end{equation}
Based on Theorem \ref{thm3.1}, we can derive the following

\begin{corollary}\label{thm4.1}
Let $u(x, t)$ be the solution of the Cauchy problem \eqref{equ4.1} with $H$ satisfying the assumption in Theorem \ref{thm3.1}.  Moreover, for any $r>0$, if $t>r^2$ and assume that
\begin{align}\label{equ4.2}
\text{supp}\,u_0\subset B(0,r/\sqrt{T_0}),\,\,\, \text{and}\,\,\,\text{supp}\,u(x, t)\subset B(0, \sqrt{T_0}\sigma t/r),
\end{align}
where $\sigma>0$ is some fixed constant and $T_0$ is some large but fixed number (see \eqref{equ3.2}). Then $u(x,t)\equiv 0$.
\end{corollary}
\begin{proof}
The proof follows immediately  by combining estimate \eqref{equ3.3} and our assumption \eqref{equ4.2}.
\end{proof}

Similarly, concerning the fractional Schr\"{o}dinger equations, Theorem \ref{thm3.2} gives
\begin{corollary}\label{thm4.2}
Let $u(x,t)$ be the solution of the Cauchy problem \eqref{equ4.1} with  $H=(-\Delta)^{\frac s2}$, $s\ge 1$. Moreover, for any $r>0$ and $t_2>t_1\ge 0$, if
\begin{align}\label{equ4.3}
\text{supp}\,u(x, t_1)\subset B(0,\, r),\,\,\, \text{and}\,\,\,\text{supp}\,u(x, t_2)\subset B(0,\, \sigma(t_2-t_1)/r^{s-1}),
\end{align}
where $\sigma>0$ is some fixed constant and $t_2-t_1>r^{s}T_0$ for some large but fixed $T_0$ (see \eqref{equ3.14}). Then $u(x, t)\equiv 0$.
\end{corollary}
\begin{remark}\label{rmk4.3}
(\romannumeral1) We comment on previous related results on unique continuation. For Schr\"{o}dinger equations with potentials $V(x,t)\in L^{(n+2)/2}(\mathbb{R}^{n+1})$, Kenig and Sogge \cite{KS} observed that if the solution vanishes in some half space of $\mathbb{R}^{n+1}$, then it vanishes identically. The key in their approach is to establish a "uniform Sobolev inequalities" for the operator $i\partial_t+\Delta$, where the elliptic case was found by Kenig, Ruiz and Sogge in \cite{KRS}, we refer to \cite{HS} and references therein for recent progress. Furthermore Ionescu and Kenig studied the following more general case
$$
(i\partial+\Delta)u=Vu+a\cdot\nabla u\,\,\, \text{on}\,\,\mathbb{R}^n\times (0,1).
$$
They proved that no non-trivial solutions can have compact support for two distinct times under suitable assumptions on $V$ and $a$. The main tool in their proof is a carleman inequality, see also \cite{IK04, KPV} and references therein in this direction.

(\romannumeral2) Our approach is quite different from the carleman inequality method mentioned above. To the best of our knowledge, this is the first time ever to discuss the close relationship between {\bf unique continuation} and {\bf minimal velocity inequalities}. We mention that other techniques (such as inverse scattering theory or analytic function approach) can be used to treat certain nonlinear Schr\"{o}dinger equations with special structures (such as integrability or analyticity), and we refer to the work of Zhang \cite{Zhang} and Bourgain \cite{B}.

(\romannumeral3) In view of the Hardy uncertainty principle and the formula \eqref{equ1.3}, stronger uniqueness results are valid by only assuming certain Gaussian type decay of the solution $e^{it\Delta}u_0$ at two different points in time, this was extended to the perturbed case $e^{-itH}$ with $H=-\Delta+V$ in a series of paper by Escauriaza, et al. \cite{EKPV1,EKPV2,EKPV3,EKPV4}. It would be interesting to know whether the commutator methods in our proof could be extended to this situation as well.
\end{remark}
Next, we consider applications to controllability for Schr\"{o}dinger type equations. Based on an abstract lemma \cite[Lemma 5.1]{WWZ} concerning the equivalence between observability and controllability, we can obtain the following result from Theorem \ref{thm3.1}.
\begin{theorem}\label{thm4.3}
Let $H$ satisfy the assumption in Theorem \ref{thm3.1}.  Consider the following impulse controlled Schr\"{o}dinger  equation  for any $r>0$, $\tau_2-\tau_1>r^2$
\begin{align}\label{equ4.4}
\begin{cases}
i\partial_{t}u -H u=\delta_{t=\tau_1}\chi(|x|\ge \frac{r}{\sqrt{T_0}})h_1+\delta_{t=\tau_2}\chi(|x|\ge\frac{\sqrt{T_0}\sigma(\tau_2-\tau_1)}{r})h_2,\,\,\,(x, t)\in \mathbb{R}^n\times (0, T),\\
u(x, 0)=u_{0}\in L^2(\mathbb{R}^n),
\end{cases}
\end{align}
where $\sigma>0$ is some fixed constant and  $T_0$ is large enough (see \eqref{equ3.2}). Denote by $u(\cdot,\cdot, u_0, h_1, h_2)$ the solution to the equation \eqref{equ4.4}. Then for any $u_0, u_T\in L^2$, there exists a pair of controls $(h_1, h_2)\in L^2\times L^2$ such that
\begin{align}\label{equ4.5}
u(x,T, u_0, h_1, h_2)= u_T,
\end{align}
and for some $C>0$
\begin{align}\label{equ4.6}
\|h_1\|^2+\|h_2\|^2\leq C\|u_T-e^{-iTH}u_0\|^2.
\end{align}
\end{theorem}
\begin{proof}
This is a direct consequence of \cite[Lemma 5.1]{WWZ}. We sketch the proof here for the sake of self-containment. Consider the following dual equation
\begin{align}\label{equ4.7}
\begin{cases}
i\partial_{t}\varphi -H\varphi=0,\,\,\,(x, t)\in \mathbb{R}^n\times (0, T),\\
u(x, T)=f\in L^2(\mathbb{R}^n),
\end{cases}
\end{align}
and denote $\varphi(\cdot, \cdot, T, f)$ the solution to \eqref{equ4.7}. Then Theorem \ref{thm3.1}  implies that
\begin{align}\label{equ4.8}
\|f\|^2\leq C\left(\int_{|x|\ge r_1}{|\varphi(\cdot, \tau_1, T, f)|^2\,dx}+\int_{|x|\ge \frac{\sigma(\tau_2-\tau_1)}{r_1}}{|\varphi(\cdot, \tau_2, T, f)|^2\,dx}\right),
\end{align}
provided $\tau_2-\tau_1>r_1^2T_0$. Now we define the state transformation operator $R: L^2\rightarrow L^2$ and the observation operator $O: L^2\rightarrow L^2\times L^2$ as follows:
\begin{align}\label{equ4.9}
Rf=f;\,\,\, Of=\left(\chi(|x|\ge r_1)\varphi(\cdot, \tau_1, T, f),\,\,\, \chi(|x|\ge \sigma(\tau_2-\tau_1)/r_1)\varphi(\cdot, \tau_2, T, f)\right).
\end{align}
Thus by \eqref{equ4.8} and \eqref{equ4.9}, we have for any $f\in L^2$,
\begin{align}\label{equ4.10}
\|Rf\|^2\leq C\|Of\|_{L^2\times L^2}^2+\frac1k\|f\|^2,\,\,\,k\in\mathbb{N}^+.
\end{align}
According to Lemma 5.1 in \cite{WWZ}, there exists a pair $(h_{1k}, h_{2k})\in L^2\times L^2$, $k\in\mathbb{N}^+$ such that the following dual inequality holds
\begin{align}\label{equ4.11}
C\|(h_{1k}, h_{2k})\|_{L^2\times L^2}^2+k\|R^*f-O^*(h_{1k}, h_{2k})\|^2\leq \|f\|^2,\,\,\,k\in\mathbb{N}^+,
\end{align}
where
\begin{align}\label{equ4.12}
R^*f=f;\,\,\, O^*(h_{1k}, h_{2k})=u(\cdot, T, 0, h_{1k}, h_{2k}).
\end{align}
Here the dual operator $O^*$ is viewed as the control operator. Then \eqref{equ4.5} and \eqref{equ4.6} are followed by choosing a weak convergence subsequence in \eqref{equ4.11} and a limiting procedure.
\end{proof}

Similarly, combining Theorem \ref{thm3.2} with Lemma 5.1 in \cite{WWZ}, we obtain the following controllability for fractional Schr\"{o}dinger equations
\begin{theorem}\label{thm4.3}
Let $H=(-\Delta)^{\frac s2}$, $s\ge 1$.  Consider the the following impulse controlled Schr\"{o}dinger equation for any $r>0$
\begin{align}\label{equ4.13}
\begin{cases}
i\partial_{t}u -H u=\delta_{t=\tau_1}\chi(|x|\ge r)h_1+\delta_{t=\tau_2}\chi(|x|\ge\sigma(\tau_2-\tau_1)/r^{s-1})h_2,\,\,\,(x, t)\in \mathbb{R}^n\times (0, T),\\
u(0,x)=u_{0}\in L^2(\mathbb{R}^n),
\end{cases}
\end{align}
where $\sigma>0$ is some fixed constant and $\tau_2-\tau_1>r^sT_0$ for some $T_0$ large enough (see \eqref{equ3.14}). Denote $u(\cdot,\cdot, u_0, h_1, h_2)$ the solution to the equation \eqref{equ4.13}. Then for any $u_0, u_T\in L^2$, there exists a pair of controls $(h_1, h_2)\in L^2\times L^2$ such that
\begin{align}\label{equ4.14}
u(x,T, u_0, h_1, h_2)= u_T,
\end{align}
and
\begin{align}\label{equ4.15}
\|h_1\|^2+\|h_2\|^2\leq C\|u_T-e^{-itH}u_0\|^2.
\end{align}
\end{theorem}

\begin{appendix}
\section{Commutator estimates}\label{sec5}
\begin{lemma}\label{lemA1}
Let $A$ and $B$ be two operators on a Hilbert space $\mathfrak{X}$ with $A$ self-adjoint and $B$ bounded. Assume that $D(A)\cap D(B)$ is dense and $[A, B]$ extends to a bounded operator. Further there is a constant $M_{AB}$ such that
$$\|[A, B]\|\leq M_{AB}.$$
Assume $0\leq \varphi\in C_0^{\infty}(\mathbb{R})$ such that $\text{supp} \varphi\subset [\frac12, 2]$, and $\varphi=1$ on $[\frac34, \frac54]$ and denote by $\varphi_N=\varphi(\frac{\cdot}{N})$. Then we have
\begin{align}\label{equA.1}
\|[\varphi_N(A), B]\|\leq CM_{AB}N^{-\frac34}.
\end{align}
\end{lemma}
\begin{proof}
Let  $g(\lambda)$ denote the Fourier transform of $\varphi_N$ and set $\psi_N=i\frac{d}{dx}\varphi_N$. Thus we have $\lambda g(\lambda)=\hat{\psi_N}$. Note that in the sense of quadratic forms on $D(A)\cap D(B)$
$$
[\varphi_N(A), B]=-i\int{g(\lambda)e^{-i\lambda A}(\int_0^{\lambda}e^{i\mu A}[A, B]e^{-i\mu A}\,d\mu)\,d\lambda}.
$$
Hence
\begin{align}\label{equA.2}
\|\left(f,\,[\varphi_N(A), B]g\right)\|&\leq M_{AB}\int_{\mathbb{R}}{|\lambda g(\lambda)|\,d\lambda}\|f\|\cdot\|g\|\nonumber\\
&\leq M_{AB}\|\psi_N\|_{\mathcal{F}L^1}\|f\|\cdot\|g\|.
\end{align}
In order to estimate the norm in \eqref{equA.2}, we apply Bernstein's inequality, i.e., $H^{\alpha}(\mathbb{R}^n)\hookrightarrow \mathcal{F}L^1(\mathbb{R}^n)$, $\alpha>\frac n2$ (see e.g., \cite[Lemma 3.2]{Hie})
\begin{align}\label{equA.3}
\|\psi_N\|_{\mathcal{F}L^1}&\leq C\|\psi_N\|_{L^2}^{\frac12}\cdot\|\frac{d}{dx}\psi_N\|_{L^2}^{\frac12}\nonumber\\
&\leq CN^{-\frac34},
\end{align}
where the constant $C$ doesn't depend on $N$. Therefore \eqref{equA.1} is followed by Combining \eqref{equA.2} and \eqref{equA.3}.
\end{proof}
\end{appendix}

\section*{Acknowledgements}
We are grateful to the anonymous referee for his/her thoughtful comments and efforts towards
improving our manuscript. S. Huang would like to thank C. Kenig for useful discussions about topics on uncertainty principle and unique continuation for Schr\"{o}dinger equations.
Part of this work was done while A. Soffer was a visiting Professor at CCNU (Central China Normal University).  The authors thank the institutions for their hospitality and the support. S. Huang is supported by the National Natural Science Foundation of China No. 11801188 and the Fundamental Research Funds for the Central Universities No. 2018KFYYXJJ041. A. Soffer is partially supported by NSFC grant No.11671163 and NSF grant DMS01600749.

\bibliographystyle{amsplain}

\end{document}